\documentclass[11pt]{amsart}

\usepackage[T1]{fontenc}
\usepackage[utf8]{inputenc}
\usepackage{lmodern}
\usepackage{microtype}
\usepackage{amsmath,amssymb,amsthm,mathtools,mathrsfs}
\usepackage{enumitem}
\usepackage{aliascnt}
\usepackage[colorlinks=true,urlcolor=blue,linkcolor=black,citecolor=blue]{hyperref}
\usepackage[nameinlink,noabbrev]{cleveref}

\allowdisplaybreaks
\numberwithin{equation}{section}

\newtheorem{theorem}{Theorem}[section]

\newaliascnt{proposition}{theorem}

\aliascntresetthe{proposition}

\newaliascnt{lemma}{theorem}
\newtheorem{lemma}[lemma]{Lemma}
\aliascntresetthe{lemma}

\newaliascnt{corollary}{theorem}
\newtheorem{corollary}[corollary]{Corollary}
\aliascntresetthe{corollary}

\newaliascnt{claim}{theorem}

\aliascntresetthe{claim}

\theoremstyle{definition}
\newaliascnt{definition}{theorem}

\aliascntresetthe{definition}

\theoremstyle{remark}
\newaliascnt{remark}{theorem}

\aliascntresetthe{remark}

\crefname{theorem}{theorem}{theorems}
\Crefname{theorem}{Theorem}{Theorems}
\crefname{proposition}{proposition}{propositions}
\Crefname{proposition}{Proposition}{Propositions}
\crefname{lemma}{lemma}{lemmas}
\Crefname{lemma}{Lemma}{Lemmas}
\crefname{corollary}{corollary}{corollaries}
\Crefname{corollary}{Corollary}{Corollaries}
\crefname{claim}{claim}{claims}
\Crefname{claim}{Claim}{Claims}
\crefname{definition}{definition}{definitions}
\Crefname{definition}{Definition}{Definitions}
\crefname{remark}{remark}{remarks}
\Crefname{remark}{Remark}{Remarks}

\newcommand{\R}{\mathbb R}
\newcommand{\Q}{\mathbb Q}
\newcommand{\Z}{\mathbb Z}
\newcommand{\C}{\mathbb C}

\newcommand{\Lio}{\mathscr L}

\newcommand{\dist}{\operatorname{dist}}
\newcommand{\diam}{\operatorname{diam}}
\newcommand{\Int}{\operatorname{int}}
\newcommand{\rank}{\operatorname{rank}}
\newcommand{\norm}[1]{\left\lVert #1\right\rVert}
\newcommand{\abs}[1]{\left\lvert #1\right\rvert}
\newcommand{\set}[1]{\left\{#1\right\}}
\newcommand{\eps}{\varepsilon}

\title[Mahler's Problem I. Analytic Rigidity]
{Mahler's Problem on Liouville Numbers
I. Analytic Rigidity}

\author{Diego Marques}

\thanks{Supported by CNPq, Grant No.~304467/2023-5.}

\address{Departamento de Matem\'atica, Universidade de Bras\'ilia, Bras\'ilia, 70910-900, Brazil}
\email{diego@mat.unb.br}
\date{}

\hypersetup{
  pdftitle={Mahler's Problem on Liouville Numbers I. Analytic Rigidity},
  pdfauthor={Diego Marques},
  pdfsubject={Analytic rigidity and Mahler's problem on Liouville numbers},
  pdfkeywords={Liouville numbers, Mahler's problem, Maillet's property,
real-analytic functions, two-height approximation, arithmetic rigidity,
irrationality exponent}
}

\begin{document}

\begin{abstract}
In 1984, Mahler asked whether a transcendental entire function
can map the set of Liouville numbers $\Lio$ into itself. We answer
this question in the negative by proving a local rigidity theorem
for real-analytic functions. More precisely, if $U\subset\R$ is
an open interval and $f:U\to\R$ is real-analytic with $f(\Lio\cap U)\subseteq\Lio$, then $f$ is the restriction to $U$ of a rational function in $\R(x)$.
Quantitatively, there exists an absolute constant $\tau>2$ such that, for
every nonrational real-analytic function $f$ and every nonempty open
subinterval $V\subset U$, the set of $\xi\in V\cap\Lio$ satisfying
$\mu(f(\xi))\leq\tau$ contains a Cantor set. The proof combines a two-height counting estimate near analytic
graphs with a nested-interval construction.
\end{abstract}

\maketitle

\setcounter{tocdepth}{1}
\tableofcontents

\section{Introduction}

A real number $\xi$ is called a \emph{Liouville number} if, for every
integer $N\geq1$, there exist infinitely many pairs
$(p,q)\in\Z\times\Z_{>0}$, with $q\geq2$, such that
\begin{equation}\label{eq:intro-liouville}
0<\abs{\xi-\frac pq}<q^{-N}.
\end{equation}
We denote the set of Liouville numbers by $\Lio$. Since their
introduction by Liouville \cite{Liouville1844,Liouville1851}, these
numbers have provided a natural testing ground for arithmetic rigidity
phenomena in analysis; see, for instance,
\cite{Bugeaud2004,FischlerRivoal2024,Waldschmidt2003}.

A classical theorem of Maillet \cite{Maillet1906} states that every
nonconstant rational function $R\in\Q(x)$ maps $\Lio$ into itself.
We say that a function has \emph{Maillet's property} if it preserves
every Liouville number in its domain. In 1984, Mahler
\cite{Mahler1984} asked whether there exist transcendental entire
functions with this property. The problem was subsequently recorded
by Waldschmidt \cite[Question~3.27, p.~281]{Waldschmidt2004} and
Bugeaud \cite[Problem~44, p.~215]{Bugeaud2004}, and has remained open
for more than four decades.

We answer Mahler's question in the negative as a consequence of a
local rigidity theorem for real-analytic functions. Recall that the
\textit{irrationality exponent} of $y\in\R$ is
\begin{equation}\label{eq:intro-mu}
\mu(y):=
\sup\left\{
\lambda>0:
\begin{aligned}
&0<\abs{y-a/b}<b^{-\lambda}\\
&\text{for infinitely many }(a,b)\in\Z\times\Z_{>0}
\end{aligned}
\right\}.
\end{equation}
Thus
\[
y\in\Lio\quad\Longleftrightarrow\quad\mu(y)=+\infty.
\]

Our main result is quantitative and local.

\begin{theorem}\label{thm:intro-local}
There exists an absolute constant $\tau>2$ with the following property.
Let $U\subseteq\R$ be an open interval and let $f:U\to\R$ be
real-analytic. Suppose that $f$ is not the restriction to $U$ of a
rational function in $\R(x)$ without poles on $U$. Then every nonempty
open subinterval $V\subseteq U$ contains a Cantor set
$K\subseteq V\cap\Lio$ such that
\[
\mu(f(\xi))\leq\tau\qquad(\xi\in K).
\]
More precisely, $K$ may be chosen together with an integer $b_0\geq2$
so that
\[
\left|f(\xi)-\frac ab\right|>b^{-\tau}
\]
for every $\xi\in K$, every $a\in\Z$, and every integer $b\geq b_0$.
One may take $\tau=100$.
\end{theorem}

The value $100$ has not been optimized. What matters is an absolute
finite bound, independent of the function and the source interval.
The cutoff $b_0$ may depend on $f$ and $V$, but is common to all points
of the constructed Cantor set.

Taking the contrapositive of Theorem~\ref{thm:intro-local} gives the
corresponding rigidity statement.

\begin{corollary}\label{cor:intro-rigidity}
Let $f:U\to\R$ be real-analytic on an open interval. If
\[
f(\Lio\cap U)\subseteq\Lio,
\]
then there exists a rational function $R\in\R(x)$, with no pole in $U$,
such that
\[
f(x)=R(x)\qquad(x\in U).
\]
\end{corollary}

The negative answer to Mahler's problem follows immediately.

\begin{theorem}\label{thm:intro-mahler}
Let $F:\C\to\C$ be entire. If
\[
F(\Lio)\subseteq\Lio,
\]
then $F\in\R[z]$. In particular, no transcendental entire function has
Maillet's property.
\end{theorem}

Indeed, since $\Lio$ is dense in $\R$, continuity implies that an entire
function preserving $\Lio$ is real-valued on $\R$.
Corollary~\ref{cor:intro-rigidity} makes its restriction to $\R$
rational, and analytic continuation together with the absence of poles
then forces the function to be a polynomial.

The hypothesis that \emph{every} Liouville number be preserved is
essential. Classical arithmetic interpolation allows great flexibility
on countable sets \cite{Stackel1895,Stackel1902,Weierstrass1923};
transcendental entire functions can even preserve every algebraic
number field \cite{VanderPoorten1968}; see also
\cite{BoxallJones2015,Lombardo2017,Surroca2006}. Preservation on large
subsets of $\Lio$ is likewise compatible with transcendental behavior:
Alnia\c{c}{\i}k and Saias \cite{AlniacikSaias1994} showed that a
continuous, nowhere locally constant real function maps a dense
$G_\delta$ subset of $\Lio$ in its domain into $\Lio$, while
transcendental entire functions can preserve distinguished subclasses
of Liouville numbers
\cite{MarquesMoreira2015,MarquesSchleischitz2016}. 
For a separate matrix analogue of Mahler's problem,
see \cite{Schleischitz2025}.

The regularity threshold is also sharp within the classical
smooth--analytic hierarchy. The flexibility known at finite orders of
smoothness \cite{LelisMoreiraSilva2026} persists in the $C^\infty$
category: there exist transcendental $C^\infty$ diffeomorphisms of
$\R$, arbitrarily close to the identity in the compact-open
$C^\infty$ topology, such that every derivative of every positive
iterate has Maillet's property \cite{MarquesSmooth2026}.
Corollary~\ref{cor:intro-rigidity} shows that this flexibility
disappears completely at the real-analytic level.

The rationality conclusion leaves a separate arithmetic question:
which nonconstant rational functions with real coefficients have
Maillet's property? In the companion manuscript \cite{MarquesII},
we give a characterization in terms of a canonical coefficient
field and the absence of Liouville points in rational fibers.
In particular, rational coefficients are necessary in degree one
but not in higher degree.

\medskip
\noindent
\textit{The two-height mechanism.}
The main challenge is to construct a Liouville number $\xi$ whose
image $f(\xi)$ has finite irrationality exponent. This requires
arbitrarily accurate rational approximations to $\xi$, while
simultaneously excluding sufficiently accurate approximations to
$f(\xi)$ at every large denominator. Our approach separates these
two arithmetic scales.

We count reduced fractions $p/q$, with $q\asymp Q$, for which
$f(p/q)$ lies exceptionally close to a rational number $a/b$.
Small target denominators are handled by elementary counting.
For $b\asymp Q^u$, we use a determinant whose source degree
adapts to $u$, while its target degree remains one.
On each short source interval, determinant vanishing forces
the corresponding rational pairs onto a relation
\[
A(x)+B(x)y=0.
\]
Since $f$ is nonrational, no such nontrivial relation vanishes
identically on its graph. Wronskian localization, uniform
sublevel estimates, and Farey separation then control the
number of exceptional source fractions.

The resulting two-height estimate, Theorem~\ref{prop:scarcity-en},
is
\[
\#\mathcal D(J;Q,A,H)
\leq C_{\mathrm{low}}Q^2H^{2-\tau}
     +C(J,A)Q^{17/10},
\]
where $A$ prescribes the source approximation order and $H$
is the lower target-denominator cutoff, with target denominators
below $Q^{A/97}$. Crucially, $C_{\mathrm{low}}$ is independent
of $J$ and $H$, and the subquadratic remainder is uniform in
$H$. These properties allow the repeated selection of rational
centers avoiding exceptionally good target approximations.
A nested-interval construction, with increasing source orders
and consecutive target-denominator ranges, then produces
$\xi\in\mathscr L$ with $\mu(f(\xi))\leq100$.

The determinant method builds on ideas of Bombieri--Pila and
Pila \cite{BombieriPila1989,Pila1991,Pila2005}; see also
\cite{BeresnevichDickinsonVelani2007,ComteYomdin2018,
Huang2015,VaughanVelani2006}.
The key feature here is the adaptive source degree combined
with a target-linear relation, yielding the uniformity needed
to accommodate unbounded approximation orders in the
nested construction.

A machine-checked formalization of
Theorems~\ref{thm:intro-local} and~\ref{thm:intro-mahler}
in the \texttt{Lean~4} proof assistant is available at
\url{https://github.com/DiegoMarquesMath/MahlerLean}.

\medskip
\noindent\emph{Acknowledgments.}
The author is grateful to Yuri Bilu, Carlos Gustavo Moreira,
Johannes Schleischitz, and Jean Lelis for many valuable discussions
related to this problem over the years. The author is especially
grateful to Michel Waldschmidt, who first introduced him to the
problem at the 2008 Arizona Winter School.

\section{Rational counting preliminaries}
\label{sec:rational-counting}

\noindent
\textbf{Notation and conventions.} For a finite union of intervals $E\subseteq\R$, we denote its Lebesgue
measure by $|E|$. We write $J\Subset U$ if $\overline J$ is a compact
subset of $U$.

The symbols $O(\cdot)$, $\ll$, and $\asymp$ have their usual meanings.
Unless otherwise stated, implied constants may depend on quantities
fixed before the estimate, but not on parameters subsequently allowed
to vary or tend to infinity. Additional dependence is indicated by
subscripts; thus $X\ll_{A,J}Y$ means that
$X\leq C(A,J)Y$ for some $C(A,J)>0$. Generic constants
$C,C_1,C_2,\ldots$ may vary from one occurrence to the next.

Subscripts and superscripts on named constants are mnemonic labels,
not indicators of parameter dependence; for example, `sub' and `low'
refer to sublevel and low-target-height estimates, respectively.

We record two standard rational-counting estimates for
Section~\ref{sec:scarcity}. The first retains the component term without
an additional factor of $Q$; the second supplies a positive proportion
of reduced fractions in every fixed interval.

\begin{lemma}[Reduced fractions in a finite union of intervals]
\label{lem:rational-union}
Let $E\subseteq\R$ be a union of at most $R$ intervals, where
$R\in\Z_{>0}$, and let $Q\in\Z_{>0}$. Reduced fractions are counted as rational
numbers, so their positive-denominator representations are unique. Then
\begin{equation}\label{eq:rational-union}
\#\set{\frac pq\in E:\gcd(p,q)=1,\ Q\leq q<2Q}
\leq4Q^2|E|+R.
\end{equation}
\end{lemma}

\begin{proof}
Write $E$ as the disjoint union of its nonempty connected components
$I_1,\ldots,I_s$, where $s\leq R$. If $p/q\neq p'/q'$ are reduced
fractions with positive denominators $q,q'<2Q$, then
\[
\abs{\frac pq-\frac{p'}{q'}}
=
\frac{|pq'-p'q|}{qq'}
\geq\frac1{qq'}>\frac1{4Q^2}.
\]
Consequently, an interval of length $\ell$ contains at most
$4Q^2\ell+1$ such fractions. Applying this to each $I_\nu$ and
summing proves \eqref{eq:rational-union}.
\end{proof}

The complementary estimate supplies a positive proportion of reduced
source fractions in every fixed interval.

\begin{lemma}[Reduced fractions in a denominator block]
\label{lem:farey-lower-en}
There exists an absolute constant $c_{\mathrm F}>0$ with the following
property. For every nondegenerate compact interval $J\subseteq\R$,
there exists $Q_0(J)\in\Z_{>0}$ such that
\begin{equation}\label{eq:farey-lower-en}
\#\set{\frac pq\in J:\gcd(p,q)=1,\ Q\leq q<2Q}
\geq c_{\mathrm F}|J|Q^2
\end{equation}
for every integer $Q\geq Q_0(J)$. One may take
$c_{\mathrm F}=1/4$.
\end{lemma}

\begin{proof}
For $q\geq1$, let $N_q(J):=\#\{p\in\Z:p/q\in J,\ \gcd(p,q)=1\}$.
M\"obius inversion gives $N_q(J)=|J|\varphi(q)+O(\tau_0(q))$, where
$\tau_0$ is the divisor function and the implied constant is absolute.
The standard summatory estimates for $\varphi$ and $\tau_0$
\cite{Apostol1976} therefore yield
\[
\sum_{Q\leq q<2Q}N_q(J)
=\frac{9}{\pi^2}|J|Q^2+O\bigl((1+|J|)Q\log Q\bigr).
\]
Since $J$ is fixed and nondegenerate and $9/\pi^2>1/4$, the assertion
follows for all sufficiently large $Q$.
\end{proof}

\section{Wronskian localization and uniform sublevel bounds}
\label{sec:wronskians}

In Section~\ref{sec:scarcity}, the determinant argument produces auxiliary
polynomials $A(X)+B(X)Y$ whose coefficients depend on the
rational points being counted. We therefore need sublevel
estimates for $A(x)+B(x)f(x)$ that are uniform over all normalized
coefficient vectors. Nonrationality of $f$ provides the required
Wronskian nondegeneracy after localization.

Let $\phi_1,\ldots,\phi_N$ be functions of class $C^{N-1}$ on an
interval. Their \textit{Wronskian} is
\begin{equation}\label{eq:wronskian-definition-en}
W(\phi_1,\ldots,\phi_N)(x)
:=
\det\bigl(\phi_j^{(k-1)}(x)\bigr)_{1\leq k,j\leq N}.
\end{equation}

\subsection{Rationality and linear independence}

We first recall the classical Wronskian criterion for analytic
functions. The analyticity assumption is essential here: the
corresponding statement fails, in general, for smooth functions.

\begin{lemma}[Analytic independence and Wronskians]
\label{lem:analytic-wronskian-en}
Let $U\subseteq\R$ be a connected open interval. If
$\phi_1,\ldots,\phi_N$ are real-analytic and linearly independent on
$U$, then
\[
W(\phi_1,\ldots,\phi_N)\not\equiv0
\qquad\text{on }U.
\]
\end{lemma}

This is the classical analytic Wronskian criterion, going back at least
to B\^ocher \cite{Bocher1900}; see also \cite{BostanDumas2010} for a modern
treatment.

For the remainder of the local argument, fix a real-analytic function
$f:U\to\R$ which is not the restriction to $U$ of a rational function
in $\R(x)$. For $d\geq0$, put
\begin{equation}\label{eq:N-d-en}
N_d:=2(d+1)
\end{equation}
and consider the indexed family
\begin{equation}\label{eq:rational-family-en}
\mathcal F_d(f)
:=
\bigl(x^i f(x)^j\bigr)_{\substack{0\leq i\leq d\\ j\in\{0,1\}}}.
\end{equation}

\begin{lemma}[Rationality and linear dependence]
\label{lem:rational-independence-en}
Let $f:U\to\R$ be real-analytic. The following are equivalent.
\begin{enumerate}[label=\textup{(\roman*)},leftmargin=2.7em]
\item $f$ is the restriction to $U$ of a rational function in
$\R(x)$ with no pole in $U$;
\item the family $\mathcal F_d(f)$ is linearly dependent over $\R$
for some $d\geq0$.
\end{enumerate}
Consequently, under the standing hypothesis on $f$, the family
$\mathcal F_d(f)$ is linearly independent for every $d\geq0$.
\end{lemma}

\begin{proof}
Suppose first that $f=P/Q$ on $U$, where $P,Q\in\R[X]$, $Q$ has no
zero in $U$, and $Q\neq0$. If
$d\geq\max\{\deg P,\deg Q\}$, then
\[
Q(x)f(x)-P(x)=0
\qquad(x\in U),
\]
which is a nontrivial linear relation among the members of
$\mathcal F_d(f)$.

Conversely, suppose that $\mathcal F_d(f)$ is linearly dependent. Then
there are polynomials $A,B\in\R[X]$, not both zero and of degree at
most $d$, such that
\begin{equation}\label{eq:A-B-f-relation-en}
A(x)+B(x)f(x)=0
\qquad(x\in U).
\end{equation}
If $B=0$, then $A$ vanishes on the open interval $U$, and hence
$A=0$, a contradiction. Thus $B\neq0$.

Let $D=\gcd(A,B)$ in $\R[X]$ and write $A=DA_0$, $B=DB_0$, with
$A_0$ and $B_0$ coprime. On the nonempty open set
$U\setminus\{D=0\}$, equation \eqref{eq:A-B-f-relation-en} gives
$A_0+B_0f=0$. Since $D$ has only finitely many zeros, continuity extends this identity to all of $U$:
\[
A_0(x)+B_0(x)f(x)=0
\qquad(x\in U).
\]
If $B_0(x_0)=0$ at some $x_0\in U$, then the last identity also gives
$A_0(x_0)=0$, contrary to the coprimality of $A_0$ and $B_0$.
Therefore $B_0$ has no zero in $U$, and
\[
f(x)=-\frac{A_0(x)}{B_0(x)}
\qquad(x\in U).
\]
This proves \textup{(i)}.
\end{proof}

Order the pairs
\[
(i,j)\in\{0,\ldots,d\}\times\{0,1\}
\]
lexicographically, and denote the corresponding ordered family of
functions $x^if(x)^j$ by
\[
\varphi_{d,1},\ldots,\varphi_{d,N_d}.
\]
Define
\begin{equation}\label{eq:W-d-en}
W_d(x)
:=
W(\varphi_{d,1},\ldots,\varphi_{d,N_d})(x).
\end{equation}
The ordering affects $W_d$ at most by a sign and therefore has no
bearing on its zero set.

\begin{corollary}\label{cor:wronskian-finite-en}
For every $d\geq0$, the function $W_d$ is real-analytic and not
identically zero on $U$. Consequently, if $I\Subset U$ is a compact
interval, then
\[
\set{x\in I:W_d(x)=0}
\]
is finite.
\end{corollary}

\begin{proof}
The functions in $\mathcal F_d(f)$ are real-analytic and, by
Lemma~\ref{lem:rational-independence-en}, linearly independent.
Lemma~\ref{lem:analytic-wronskian-en} therefore gives
$W_d\not\equiv0$ on $U$. The function $W_d$ is real-analytic because
it is a determinant of derivatives of real-analytic functions. If it
had infinitely many zeros in $I$, compactness would give an
accumulation point in $U$, and the identity theorem would force
$W_d\equiv0$ on the connected interval $U$, a contradiction.
\end{proof}

\subsection{One-dimensional sublevel estimates}

We next record the elementary one-dimensional estimate underlying the
uniform argument. It expresses the principle that a function with a
derivative uniformly separated from zero cannot remain small on a
long set.

\begin{lemma}\label{lem:one-derivative-en}
Let $k\geq1$, let $I\subseteq\R$ be a nondegenerate compact interval,
and let $g\in C^k(I)$. Suppose that
\begin{equation}\label{eq:kth-lower-en}
\abs{g^{(k)}(x)}\geq\lambda>0
\qquad(x\in I).
\end{equation}
Then, for every $\eps>0$, the set
\[
E_g(\eps):=\set{x\in I:\abs{g(x)}\leq\eps}
\]
has at most $2k+1$ connected components and satisfies
\begin{equation}\label{eq:one-derivative-length-en}
|E_g(\eps)|
\leq
2k(2k+1)
\left(\frac{\eps}{\lambda}\right)^{1/k}.
\end{equation}
\end{lemma}

\begin{proof}
Since $g^{(k)}$ is continuous and nowhere zero on $I$, it has constant
sign. By Rolle's theorem, for every $c\in\R$ the equation $g(x)=c$
has at most $k$ solutions. Hence the relative boundary of
\[
E_g(\varepsilon)=\{x\in I:|g(x)|\leq\varepsilon\}
\]
has at most $2k$ points, and $E_g(\varepsilon)$ has at most
$2k+1$ connected components.

Let $[a,b]$ be a nondegenerate component of $E_g(\varepsilon)$, and put
\[
h:=\frac{b-a}{k},
\qquad
x_j:=a+jh
\quad(0\leq j\leq k).
\]
By the mean-value theorem for finite differences, for some
$\theta\in(a,b)$,
\[
h^k g^{(k)}(\theta)
=
\sum_{j=0}^k(-1)^{k-j}\binom{k}{j}g(x_j).
\]
Since $|g(x_j)|\leq\varepsilon$,
\[
h^k\lambda
\leq
2^k\varepsilon,
\]
and therefore
\[
b-a=kh
\leq
2k\left(\frac{\varepsilon}{\lambda}\right)^{1/k}.
\]
Summing over at most $2k+1$ nondegenerate components gives
\[
|E_g(\varepsilon)|
\leq
2k(2k+1)
\left(\frac{\varepsilon}{\lambda}\right)^{1/k}.
\]
\end{proof}

We now pass from a lower bound for one prescribed derivative to a
bound that is uniform over all normalized linear combinations of a
fixed finite family. Wronskian nondegeneracy gives a uniform lower bound for at least
one derivative at each point. Compactness of the product of the
coefficient sphere and the source interval makes the resulting
local estimates uniform.

\begin{lemma}[Uniform sublevel bounds under Wronskian nondegeneracy]
\label{lem:uniform-sublevel-en}
Let $J\subseteq\R$ be a nondegenerate compact interval, let $N\geq2$,
and let $\phi_1,\ldots,\phi_N$ be of class $C^{N-1}$ on a neighborhood
of $J$. Assume that
\begin{equation}\label{eq:wronskian-separated-en}
\inf_{x\in J}\abs{W(\phi_1,\ldots,\phi_N)(x)}>0.
\end{equation}
For $c=(c_1,\ldots,c_N)\in\R^N$, write
\[
G_c(x)=\sum_{j=1}^N c_j\phi_j(x).
\]
Then there exist constants
\[
C_{\mathrm{sub}}>0,
\qquad
R_{\mathrm{sub}}\in\Z_{>0},
\qquad
\eps_0\in(0,1)
\]
such that, for every $c\in\R^N$ with $\norm{c}_2=1$ and every
$0<\eps<\eps_0$, the set
\[
\set{x\in J:\abs{G_c(x)}\leq\eps}
\]
is a union of at most $R_{\mathrm{sub}}$ intervals and satisfies
\begin{equation}\label{eq:uniform-sublevel-en}
\left|\set{x\in J:\abs{G_c(x)}\leq\eps}\right|
\leq C_{\mathrm{sub}}\eps^{1/(N-1)}.
\end{equation}
\end{lemma}

\begin{proof}
Consider the jet matrix
\[
\mathcal A(x)
=
\bigl(\phi_j^{(k)}(x)\bigr)_{
\substack{0\leq k\leq N-1\\1\leq j\leq N}}.
\]
Its determinant is the Wronskian in
\eqref{eq:wronskian-separated-en}; hence $\mathcal A(x)$ is invertible
for every $x\in J$. The least singular value of $\mathcal A(x)$
depends continuously on $x$. Since $J$ is compact, there exists
$\sigma>0$ such that
\begin{equation}\label{eq:uniform-singular-value-en}
\norm{\mathcal A(x)c}_2\geq\sigma\norm{c}_2
\qquad(x\in J,\ c\in\R^N).
\end{equation}
For $\norm{c}_2=1$,
\[
\mathcal A(x)c
=
\bigl(G_c(x),G_c'(x),\ldots,G_c^{(N-1)}(x)\bigr)^{\mathsf T},
\]
and therefore
\[
\max_{0\leq k\leq N-1}\abs{G_c^{(k)}(x)}
\geq
\eta,
\qquad
\eta:=\frac{\sigma}{\sqrt N},
\]
for every $(c,x)\in\mathbb S^{N-1}\times J$.

Fix $(c_0,x_0)\in\mathbb S^{N-1}\times J$. There is an index
$k(c_0,x_0)\in\{0,\ldots,N-1\}$ for which
\[
\abs{G_{c_0}^{(k(c_0,x_0))}(x_0)}\geq\eta.
\]
Continuity in $(c,x)$ yields a relatively open neighborhood $U$ of
$c_0$ in $\mathbb S^{N-1}$ and a compact interval $I\subseteq J$,
with $x_0\in\Int_J(I)$, such that
\[
\abs{G_c^{(k(c_0,x_0))}(x)}\geq\frac{\eta}{2}
\qquad(c\in U,\ x\in I).
\]
These sets $U\times\Int_J(I)$ cover the compact space
$\mathbb S^{N-1}\times J$. Choose a finite subcover
\[
U_\nu\times\Int_J(I_\nu),
\qquad1\leq\nu\leq L,
\]
and let $k_\nu\in\{0,\ldots,N-1\}$ denote the corresponding derivative
order. Thus
\begin{equation}\label{eq:local-large-jet-en}
\abs{G_c^{(k_\nu)}(x)}\geq\frac{\eta}{2}
\qquad(c\in U_\nu,\ x\in I_\nu).
\end{equation}
The finite cover, and hence $L$ and all constants derived from it,
depend only on $J$ and the family $\phi_1,\ldots,\phi_N$, and not on
the subsequent choice of $c\in\mathbb S^{N-1}$.

Fix $c\in\mathbb S^{N-1}$ and let
\[
\mathcal N(c):=\set{\nu\in\{1,\ldots,L\}:c\in U_\nu}.
\]
Then the intervals $I_\nu$, $\nu\in\mathcal N(c)$, cover $J$ by their
relative interiors. Put
\[
\eps_0:=\min\left\{\frac12,\frac{\eta}{4}\right\}.
\]
If $0<\eps<\eps_0$ and $k_\nu=0$, then
\eqref{eq:local-large-jet-en} shows that the sublevel set does not meet
$I_\nu$. If $1\leq k_\nu\leq N-1$, Lemma~\ref{lem:one-derivative-en},
with $\lambda=\eta/2$, gives
\[
\left|\set{x\in I_\nu:\abs{G_c(x)}\leq\eps}\right|
\leq
2k_\nu(2k_\nu+1)
\left(\frac{2\eps}{\eta}\right)^{1/k_\nu}.
\]
Since $0<\eps<1$ and $k_\nu\leq N-1$, this is at most
\[
C_*\eps^{1/(N-1)},
\qquad
C_*:=
\max_{1\leq k\leq N-1}
2k(2k+1)\left(\frac{2}{\eta}\right)^{1/k}.
\]
The same lemma shows that each such local sublevel set has at most
$2N-1$ connected components.

Since $\#\mathcal N(c)\leq L$, subadditivity of length now gives
\[
\left|\set{x\in J:\abs{G_c(x)}\leq\eps}\right|
\leq
LC_*\eps^{1/(N-1)}.
\]
Moreover, this set has at most $L(2N-1)$ connected components.
Thus one may take
\[
C_{\mathrm{sub}}:=LC_*,
\qquad
R_{\mathrm{sub}}:=L(2N-1),
\]
and these constants, as well as $\eps_0$, depend only on $J$ and the
family $\phi_1,\ldots,\phi_N$, and not on $c$.
\end{proof}

We specialize this result to the target-linear polynomials that arise
from the determinant argument. For
\[
P(X,Y)=A(X)+B(X)Y
=\sum_{i=0}^d c_{i0}X^i+\sum_{i=0}^d c_{i1}X^iY,
\]
write
\[
\norm{P}_2
:=
\left(\sum_{i=0}^d(c_{i0}^2+c_{i1}^2)\right)^{1/2}.
\]

\begin{corollary}[Uniform rational-relation sublevel estimate]
\label{cor:polynomial-sublevel-en}
Fix $d\geq0$ and a nondegenerate compact interval $J\Subset U$ on which $W_d$ does
not vanish. Put $N=N_d=2(d+1)$. Then there exist constants
\[
C_{d,J}>0,
\qquad
R_{d,J}\in\Z_{>0},
\qquad
\eps_{d,J}>0
\]
such that, for every polynomial
\[
P(X,Y)=A(X)+B(X)Y,
\qquad
\deg A,\deg B\leq d,
\qquad
\norm{P}_2=1,
\]
and every $0<\eps<\eps_{d,J}$, the set
\[
\set{x\in J:\abs{P(x,f(x))}\leq\eps}
\]
is a union of at most $R_{d,J}$ intervals and satisfies
\begin{equation}\label{eq:poly-sublevel-en}
\left|\set{x\in J:\abs{P(x,f(x))}\leq\eps}\right|
\leq C_{d,J}\eps^{1/(N-1)}.
\end{equation}
The constants depend only on $d$, $J$, and the restriction of $f$ to a
neighborhood of $J$.
\end{corollary}

\begin{proof}
Arrange the coefficients $(c_{i0},c_{i1})$ in the same order as the
family used to define $W_d$. Then $P(x,f(x))$ is the corresponding
linear combination of the members of $\mathcal F_d(f)$, and its
coefficient vector has Euclidean norm one. Since $W_d$ is continuous,
nonzero on the compact interval $J$, and hence bounded away from zero,
Lemma~\ref{lem:uniform-sublevel-en} applies and gives the assertion.
\end{proof}

\section{Anisotropic determinant estimates}
\label{sec:determinant}

We compare determinants formed from rational points close to
the graph of $f$. Smoothness bounds these determinants from
above on short source intervals, while clearing the source and
target denominators separately gives a lower bound for every
nonzero determinant. Their comparison will force an auxiliary
relation of degree one in the target variable.

For vectors $v_1,\ldots,v_s\in\R^N$, we equip the exterior power
$\bigwedge^s\R^N$ with its Euclidean norm. Thus
\[
\norm{v_1\wedge\cdots\wedge v_s}^2
=
\det\bigl(\langle v_i,v_j\rangle\bigr)_{1\leq i,j\leq s}.
\]
In particular, when $s=N$,
\[
\norm{v_1\wedge\cdots\wedge v_N}
=
\abs{\det(v_1,\ldots,v_N)}.
\]
We shall also use
\[
\norm{v\wedge\omega}\leq\norm{v}\,\norm{\omega}
\]
for $v\in\R^N$ and $\omega\in\bigwedge^s\R^N$.

\subsection{Exterior decay along a smooth curve}

Several values of a smooth one-dimensional curve taken in an interval
of length $\rho$ become increasingly dependent as $\rho\to0$. The
exponent below reflects the successive Taylor orders
$0,1,\ldots,s-1$.

\begin{lemma}\label{lem:wedge-curve-en}
Let $J\subseteq\R$ be a nondegenerate compact interval, let
$1\leq s\leq N$, and let
\[
\Phi:J\longrightarrow\R^N
\]
be the restriction to $J$ of a $C^s$ map defined on a neighborhood of
$J$. There exists a constant $C_\Phi>0$ such that, whenever
$0<\rho\leq1$ and $t_1,\ldots,t_s$ lie in a common subinterval of $J$
of length at most $\rho$, one has
\begin{equation}\label{eq:wedge-bound-en}
\norm{\Phi(t_1)\wedge\cdots\wedge\Phi(t_s)}
\leq C_\Phi\rho^{s(s-1)/2}.
\end{equation}
\end{lemma}

\begin{proof}
Let
\[
a:=\min_{1\leq i\leq s}t_i,
\qquad
\ell:=\max_{1\leq i\leq s}t_i-a.
\]
If $\ell=0$, the assertion is immediate: for $s\geq2$ the exterior
product vanishes, while for $s=1$ it follows after taking
$C_\Phi\geq\sup_J\norm{\Phi}$. Assume henceforth that
$0<\ell\leq\rho$. Taylor's formula with integral remainder gives
\begin{equation}\label{eq:vector-taylor-en}
\Phi(t_i)
=
\sum_{r=0}^{s-1}(t_i-a)^rv_r+R_i,
\qquad
v_r:=\frac{\Phi^{(r)}(a)}{r!},
\end{equation}
where
\[
R_i
=
\frac{1}{(s-1)!}
\int_a^{t_i}(t_i-u)^{s-1}\Phi^{(s)}(u)\,du.
\]
Since $\Phi^{(s)}$ is bounded on a neighborhood of $J$, there is a
constant $C_0>0$, independent of $a$, $\ell$, and $i$, such that
\begin{equation}\label{eq:remainder-bound-en}
\norm{R_i}\leq C_0\abs{t_i-a}^s\leq C_0\ell^s.
\end{equation}
Also $\abs{t_i-a}^r\leq\ell^r$ for $0\leq r\leq s-1$.

Expand the exterior product of the expressions in
\eqref{eq:vector-taylor-en}. Consider a term containing exactly $k$
remainder vectors. The remaining $s-k$ factors are chosen from
$v_0,\ldots,v_{s-1}$. A term in which some $v_r$ occurs twice vanishes.
Thus, in every nonzero term, the corresponding indices
$r_1,\ldots,r_{s-k}$ are distinct and
\[
r_1+\cdots+r_{s-k}
\geq
0+1+\cdots+(s-k-1)
=
\frac{(s-k)(s-k-1)}{2}.
\]
Such a term is therefore bounded by a constant times
\[
\ell^{ks+(s-k)(s-k-1)/2}.
\]
The identity
\[
ks+\frac{(s-k)(s-k-1)}{2}
=
\frac{s(s-1)}{2}+\frac{k(k+1)}{2}
\]
shows that this is
$O_\Phi(\ell^{s(s-1)/2})$. There are only finitely many terms, and the
vectors $v_0,\ldots,v_{s-1}$ are uniformly bounded as $a$ ranges over
$J$. Hence
\[
\norm{\Phi(t_1)\wedge\cdots\wedge\Phi(t_s)}
\leq C_\Phi\ell^{s(s-1)/2}
\leq C_\Phi\rho^{s(s-1)/2}.
\]
\end{proof}

\subsection{The perturbed determinant}

For $d\geq0$, put $N=2(d+1)$, order the pairs
\[
(i,j)\in\{0,\ldots,d\}\times\{0,1\}
\]
as in Section~\ref{sec:wronskians}, and define
\begin{equation}\label{eq:Psi-d-en}
\Psi_d(x,y):=\bigl(x^iy^j\bigr)_{
\substack{0\leq i\leq d\\j\in\{0,1\}}}\in\R^N,
\qquad
\Phi_d(x):=\Psi_d(x,f(x)).
\end{equation}

\begin{lemma}[Perturbed target-linear determinant]
\label{lem:perturbed-determinant-en}
Let $d\geq0$, put $N=2(d+1)$, and let $J\Subset U$ be a nondegenerate compact
interval. There exists a constant $C_{d,J}^{\mathrm{pert}}>0$, depending
only on $d$, $J$, and the restriction of $f$ to a neighborhood of $J$,
with the following property. Suppose that $0<\rho\leq1$, that
$x_1,\ldots,x_N\in J$ lie in a common subinterval of $J$ of length at
most $\rho$, and that $y_1,\ldots,y_N\in\R$ satisfy
\begin{equation}\label{eq:vertical-error-en}
\abs{y_k-f(x_k)}\leq\delta\leq\rho^N
\qquad(1\leq k\leq N).
\end{equation}
Regarding the vectors $\Psi_d(x_k,y_k)$ as the rows of an
$N\times N$ matrix, one has
\begin{equation}\label{eq:perturbed-simple-en}
\abs{\det\bigl(\Psi_d(x_k,y_k)\bigr)_{1\leq k\leq N}}
\leq
C_{d,J}^{\mathrm{pert}}\rho^{N(N-1)/2}.
\end{equation}
\end{lemma}

\begin{proof}
Set
\[
E_k:=\Psi_d(x_k,y_k)-\Psi_d(x_k,f(x_k)).
\]
The coordinates of $E_k$ corresponding to $j=0$ vanish, whereas those
corresponding to $j=1$ are
\[
x_k^i\bigl(y_k-f(x_k)\bigr),
\qquad0\leq i\leq d.
\]
Since $x_k\in J$, there is a constant $C_0=C_0(d,J)$ such that
\begin{equation}\label{eq:perturbation-vector-en}
\norm{E_k}\leq C_0\delta
\qquad(1\leq k\leq N).
\end{equation}
We have
\[
\Psi_d(x_k,y_k)=\Phi_d(x_k)+E_k.
\]
Expanding the exterior product of these $N$ vectors by multilinearity,
consider a term containing exactly $r$ error vectors. If $r<N$, the
exterior norm inequality, \eqref{eq:perturbation-vector-en}, and
Lemma~\ref{lem:wedge-curve-en}, applied to the remaining $N-r$ curve
vectors, bound this term by
\[
C_r\delta^r\rho^{(N-r)(N-r-1)/2}.
\]
For $r=N$, the same bound holds with the standard convention
$\bigwedge^0\R^N=\R$ and $\norm{1}=1$. Absorbing the finitely many
combinatorial factors, we obtain
\begin{equation}\label{eq:perturbed-general-en}
\abs{\det\bigl(\Psi_d(x_k,y_k)\bigr)_{1\leq k\leq N}}
\leq
C\sum_{r=0}^N
\delta^r\rho^{(N-r)(N-r-1)/2}.
\end{equation}
By \eqref{eq:vertical-error-en}, the summand of index $r$ is at most
\[
\rho^{rN+(N-r)(N-r-1)/2}.
\]
The exact identity
\begin{equation}\label{eq:exponent-identity-en}
rN+\frac{(N-r)(N-r-1)}{2}
=
\frac{N(N-1)}{2}+\frac{r(r+1)}{2}
\end{equation}
shows that, since $0<\rho\leq1$, every summand is at most
$\rho^{N(N-1)/2}$. Summing over $0\leq r\leq N$ proves
\eqref{eq:perturbed-simple-en}.
\end{proof}

\subsection{The arithmetic lower bound}

The preceding estimate is purely analytic. Rationality supplies the
complementary fact that a nonzero target-linear determinant cannot be
too small.

\begin{lemma}\label{lem:det-lower-en}
Fix $d\geq0$ and put $N=2(d+1)$. Let
\[
x_k=\frac{p_k}{q_k},
\qquad
y_k=\frac{a_k}{b_k},
\qquad1\leq k\leq N,
\]
where $p_k,a_k\in\Z$ and $q_k,b_k\in\Z_{>0}$. The fractions need not be
reduced. Define
\[
\Delta
:=
\det\bigl(x_k^iy_k^j\bigr)_{
\substack{1\leq k\leq N\\0\leq i\leq d,\ j\in\{0,1\}}}.
\]
If $\Delta\neq0$, then
\begin{equation}\label{eq:det-lower-product-en}
\abs{\Delta}\geq\prod_{k=1}^Nq_k^{-d}b_k^{-1}.
\end{equation}
In particular, if $q_k<2Q$ and $b_k<2B$ for every $k$, then
\begin{equation}\label{eq:det-lower-dyadic-en}
\abs{\Delta}\geq(2Q)^{-dN}(2B)^{-N}.
\end{equation}
\end{lemma}

\begin{proof}
Multiply the $k$th row of the matrix defining $\Delta$ by $q_k^db_k$.
The entry corresponding to $(i,j)$ becomes
\[
q_k^db_kx_k^iy_k^j
=
\begin{cases}
 p_k^iq_k^{d-i}b_k,&j=0,\\
 p_k^iq_k^{d-i}a_k,&j=1,
\end{cases}
\]
and is therefore an integer. The determinant of the resulting integer
matrix is
\[
\left(\prod_{k=1}^Nq_k^db_k\right)\Delta.
\]
If $\Delta\neq0$, this is a nonzero integer and hence has absolute
value at least one, proving \eqref{eq:det-lower-product-en}. The final assertion follows from $q_k<2Q$ and $b_k<2B$.
\end{proof}

The next section combines the analytic upper bound
\eqref{eq:perturbed-simple-en} with the arithmetic lower bound
\eqref{eq:det-lower-dyadic-en}. The two estimates depend separately on the source and target heights
and are balanced by adapting the source degree to their relative size.

\section{Two-height counting on analytic graphs}
\label{sec:scarcity}

We now establish the main counting estimate for the nonrational
real-analytic function $f:U\to\R$ fixed in
Section~\ref{sec:wronskians}. We count rational source points
$p/q$, with $q\asymp Q$, whose images admit exceptionally good
rational approximations, keeping the source and target denominator
scales separate. The numerical parameters below are chosen to
ensure the strict exponent inequalities needed in the proof;
we make no attempt to optimize them.

We fix
\begin{equation}\label{eq:fixed-parameters-en}
\gamma=10,
\qquad
\tau=100,
\qquad
v=\frac{2}{\gamma}=\frac15,
\qquad
\tau-3=97.
\end{equation}
The target degree is one throughout this section. The parameter
$\gamma$ controls the source degree, $\tau$ is the approximation
exponent that will eventually be excluded for the image, and $v$
separates the elementary and determinant ranges of target
denominators.

For every integer $A\geq3$, define
\begin{equation}\label{eq:D-A-en}
D(A)
:=
\max\set{2,\left\lceil\frac{\gamma A}{\tau-3}\right\rceil}
=
\max\set{2,\left\lceil\frac{10A}{97}\right\rceil},
\end{equation}
and, for $Q\geq2$, put
\begin{equation}\label{eq:T-Q-A-en}
T(Q,A):=Q^{A/(\tau-3)}=Q^{A/97}.
\end{equation}
The quantity $A$ measures the required accuracy in the source
variable, whereas $T(Q,A)$ is the largest target height treated at
source scale $Q$. The exponent $\tau-3$ is chosen for the subsequent fusion: for fixed $A$,
$T(Q,A)^{2-\tau}=o(Q^{-A})$ as $Q\to\infty$.
Thus the contribution from the next lower target cutoff can
be made small relative to the next source interval.

Let $I_*\Subset U$ be a nondegenerate compact interval on which
\begin{equation}\label{eq:derivative-bounds-en}
0<m_*\leq\abs{f'(x)}\leq M_*
\qquad(x\in I_*).
\end{equation}
Since $f'$ is continuous and does not vanish on $I_*$, it has constant
sign there. Hence $f|_{I_*}$ is strictly monotone, and the mean-value
theorem gives
\begin{equation}\label{eq:inverse-lipschitz-en}
\abs{x_2-x_1}
\leq
m_*^{-1}\abs{f(x_2)-f(x_1)}
\qquad(x_1,x_2\in I_*).
\end{equation}

For a nondegenerate compact interval $J\subseteq I_*$, integers
$A\geq3$ and $H\geq2$, and an integer $Q\geq2$, define the set of \textit{resonant source centers}
\begin{equation}\label{eq:dangerous-set-en}
\begin{split}
\mathcal D(J;Q,A,H)
:=
\Bigl\{\frac pq\in J:\;&
\gcd(p,q)=1,\quad Q\leq q<2Q,\\
&\exists\,a\in\Z,\ b\in\Z_{>0}
\text{ with }H\leq b<T(Q,A),\\
&\abs{f(p/q)-a/b}
\leq2b^{-\tau}+4M_*Q^{-A}
\Bigr\}.
\end{split}
\end{equation}

The source fractions are reduced, whereas the target fractions need not be.
For $p/q\in\mathcal D(J;Q,A,H)$, any pair $(a,b)$ satisfying the
conditions in \eqref{eq:dangerous-set-en} will be called a \emph{witness}
for $p/q$.
The term $Q^{-A}$ is a stability margin which will allow avoidance at a
rational center to persist on the next source interval.

\begin{theorem}[Anisotropic two-height counting]
\label{prop:scarcity-en}
There exists a constant $C_{\mathrm{low}}>0$, depending only on
$f$, $I_*$, $m_*$, $M_*$ and the fixed parameters in
\eqref{eq:fixed-parameters-en}, with the following property. Let
$J\subseteq I_*$ be a nondegenerate compact interval and let $A\geq3$
be an integer. Assume that
\begin{equation}\label{eq:wronskians-on-J-en}
W_d(x)\neq0
\qquad(x\in J,\ 2\leq d\leq D(A)).
\end{equation}
Then there exist constants
\[
C(J,A)>0
\qquad\text{and}\qquad
Q_0(J,A)\in\Z_{>0}
\]
such that, for every integer $H\geq2$ and every integer
$Q\geq Q_0(J,A)$,
\begin{equation}\label{eq:scarcity-bound-en}
\#\mathcal D(J;Q,A,H)
\leq
C_{\mathrm{low}}Q^2H^{-98}
+
C(J,A)Q^{17/10}.
\end{equation}
In particular, the constants $C(J,A)$ and $Q_0(J,A)$ are uniform in
$H$.
\end{theorem}

\begin{proof}
Fix $J$ and $A$ as in the statement. Every lower bound imposed on $Q$
below will depend at most on $J$ and $A$, and will be independent of
$H$ and of the target block.

If $T(Q,A)\leq H$, then
$\mathcal D(J;Q,A,H)=\varnothing$. Assume henceforth that
$T(Q,A)>H$. For every integer $k\geq0$ such that
\[
B_k:=2^kH<T(Q,A),
\]
put
\begin{equation}\label{eq:dyadic-blocks-en}
\mathcal B_k
:=
[B_k,\min\{2B_k,T(Q,A)\})\cap\Z_{>0}.
\end{equation}
These sets form a disjoint partition of
$[H,T(Q,A))\cap\Z_{>0}$. Let $\mathcal D_k$ be the set of source centers
admitting at least one
witness $(a,b)$ with
$b\in\mathcal B_k$. Then
\begin{equation}\label{eq:D-covered-by-blocks-en}
\mathcal D(J;Q,A,H)\subseteq\bigcup_k\mathcal D_k,
\qquad
\#\mathcal D(J;Q,A,H)\leq\sum_k\#\mathcal D_k.
\end{equation}
A center may occur in more than one $\mathcal D_k$, which only produces
an admissible overcount.

A block with left endpoint $B=B_k$ is called \emph{small} if
$B<Q^v$ and \emph{large} if $B\geq Q^v$. A block beginning below
$Q^v$ and crossing that threshold is assigned to the small range; the
estimates there use only $B\leq b<2B$.

\subsubsection*{Small target heights}

Set
\begin{equation}\label{eq:F-star-C-star-en}
F_*:=\max_{x\in I_*}\abs{f(x)},
\qquad
C_*:=F_*+2^{1-\tau}+\frac{M_*}{2}.
\end{equation}
Suppose that $x\in I_*$, $b\geq2$, $A\geq3$, $Q\geq2$, and
\[
\abs{f(x)-a/b}\leq2b^{-\tau}+4M_*Q^{-A}.
\]
Since $Q^{-A}\leq2^{-3}$,
\[
\frac{|a|}{b}
\leq
F_*+2b^{-\tau}+4M_*Q^{-A}
\leq C_*.
\]
Thus $|a|\leq C_*b$. For $B\leq b<2B$, there are therefore
$O(B^2)$ pairs $(a,b)$ that can occur as witnesses, with an implied constant depending only on $C_*$.

For each such pair, let
\[
E_{a,b}
:=
\set{x\in J:
\abs{f(x)-a/b}\leq2b^{-\tau}+4M_*Q^{-A}}.
\]
Because $f$ is strictly monotone on $I_*$, the set $E_{a,b}$ is empty,
a singleton, or an interval. If $x_1<x_2$ belong to $E_{a,b}$, then
\eqref{eq:inverse-lipschitz-en} gives
\[
m_*(x_2-x_1)
\leq
\abs{f(x_2)-f(x_1)}
\leq
2\bigl(2b^{-\tau}+4M_*Q^{-A}\bigr).
\]
Consequently,
\begin{equation}\label{eq:Eab-length-en}
|E_{a,b}|
\leq
\frac{2}{m_*}\bigl(2b^{-\tau}+4M_*Q^{-A}\bigr).
\end{equation}

Let $E_B$ be the union of the sets $E_{a,b}$ over one small target
block. It is a union of $O(B^2)$ intervals, and
\begin{equation}\label{eq:low-target-length-en}
|E_B|\ll B^{2-\tau}+B^2Q^{-A}.
\end{equation}
Every center in the corresponding $\mathcal D_k$ belongs to $E_B$.
Applying Lemma~\ref{lem:rational-union} gives
\begin{align}
\#\mathcal D_k
&\ll Q^2|E_B|+B^2\notag\\
&\ll
Q^2B^{2-\tau}
+Q^{2-A}B^2
+B^2.
\label{eq:low-block-count-en}
\end{align}

We sum over the dyadic left endpoints $B=2^kH<Q^v$. Since
$2-\tau<0$,
\[
\sum_{2^kH<Q^v}(2^kH)^{2-\tau}
\ll H^{2-\tau},
\]
whereas
\[
\sum_{2^kH<Q^v}(2^kH)^2\ll Q^{2v}.
\]
It follows that
\begin{equation}\label{eq:low-total-en}
\sum_{B_k<Q^v}\#\mathcal D_k
\leq
C_{\mathrm{low}}Q^2H^{2-\tau}
+O\bigl(Q^{2-A+2v}+Q^{2v}\bigr).
\end{equation}
The coefficient $C_{\mathrm{low}}$ depends only on
$f$, $I_*$, $m_*$, $M_*$ and the fixed parameters in
\eqref{eq:fixed-parameters-en}. Since $A\geq3$ and $v=1/5$,
\[
2-A+2v\leq-\frac35,
\qquad
2v=\frac25.
\]
Thus
\begin{equation}\label{eq:low-final-en}
\sum_{B_k<Q^v}\#\mathcal D_k
\leq
C_{\mathrm{low}}Q^2H^{2-\tau}+O(Q^{2/5}).
\end{equation}

\subsubsection*{Large target heights}

If there are no blocks with $B_k\geq Q^v$, there is nothing to prove. This is automatically the case when $A\leq 19$, since then
\[
T(Q,A)=Q^{A/97}<Q^{1/5}=Q^v.
\]
Hence, in the large-target range one necessarily has $A\geq 20$. Fix one such block, write $B=B_k$, and define
\begin{equation}\label{eq:u-definition-en}
u:=\frac{\log B}{\log Q},
\qquad B=Q^u.
\end{equation}
Since $Q^v\leq B<T(Q,A)$,
\begin{equation}\label{eq:u-range-en}
\frac15=v\leq u<\frac{A}{97}.
\end{equation}
We shall show that each large target block contributes 
$O_{J,A}(Q^{33/20})$ resonant centers: we partition $J$ into
$O_J(Q^\kappa)$ short source cells and prove that each cell contains
only $O_{J,A}(1)$ resonant centers.
Set
\begin{equation}\label{eq:d-N-kappa-en}
d:=\lceil\gamma u\rceil=\lceil10u\rceil,
\qquad
N:=2(d+1),
\qquad
\kappa:=\frac{3(d+u)}{N-1},
\end{equation}
and put
\begin{equation}\label{eq:rho-definition-en}
\rho:=Q^{-\kappa}.
\end{equation}
The degree $d\asymp u$ adapts the source degree to the relative target
height $B=Q^u$, while the choice of $\kappa$ gives
\[
\rho^{N(N-1)/2}=Q^{-\frac32N(d+u)},
\]
to be compared below with the arithmetic scale $Q^{-N(d+u)}$.
Because $10u\geq2$, we have $d\geq2$. Moreover,
$10u<10A/97$, and hence $d\leq D(A)$. Thus
\eqref{eq:wronskians-on-J-en} applies to the chosen degree.

Partition $J$ into half-open cells of length at most $\rho$, taking the
final cell to be closed at its right endpoint. Every point of $J$ then
belongs to exactly one cell, and the number of cells is at most
\begin{equation}\label{eq:number-cells-en}
1+|J|\rho^{-1}
=1+|J|Q^\kappa
\leq(1+|J|)Q^\kappa.
\end{equation}

Fix a cell $C$. For every center $x=p/q\in\mathcal D_k\cap C$, choose once and for all
a witness $(a,b)$ with $b\in\mathcal B_k$, and put $y=a/b$. Define
\begin{equation}\label{eq:s-definition-en}
s:=\min\{u\tau,A\}.
\end{equation}
Since $b\geq B=Q^u$,
\begin{align}
\abs{y-f(x)}
&\leq2B^{-\tau}+4M_*Q^{-A}\notag\\
&\leq K_0Q^{-s},
\qquad K_0:=2+4M_*.
\label{eq:vertical-Q-s-en}
\end{align}

We record uniform exponent bounds. From $d=\lceil10u\rceil$ and
$u\geq1/5$,
\begin{equation}\label{eq:basic-d-bounds-en}
10u\leq d\leq10u+1\leq15u,
\qquad
d+u\leq16u.
\end{equation}
Since $u\leq d/10$ and $N-1=2d+1$,
\begin{equation}\label{eq:kappa-bound-en}
\kappa
=\frac{3(d+u)}{2d+1}
\leq\frac{3(11d/10)}{2d}
=\frac{33}{20}.
\end{equation}
The upper bound in \eqref{eq:u-range-en} gives $97u<A$, while
$\tau u=100u$. Therefore
\begin{equation}\label{eq:s-lower-en}
s\geq97u.
\end{equation}
Also, using $N/(N-1)\leq2$ and
\eqref{eq:basic-d-bounds-en},
\begin{equation}\label{eq:kappaN-upper-en}
\kappa N
=3(d+u)\frac{N}{N-1}
\leq6(d+u)
\leq96u.
\end{equation}
Consequently,
\begin{equation}\label{eq:s-kappa-gap-en}
s-\kappa N\geq u\geq\frac{1}{5}.
\end{equation}
Finally,
\begin{equation}\label{eq:N-upper-en}
N-1=2d+1\leq30u+1\leq35u.
\end{equation}
Thus the three numerical inequalities driving the large-target
argument are
\[
\kappa\leq\frac{33}{20},\qquad
s-\kappa N\geq\frac15,\qquad
\frac{s}{N-1}\geq\frac{97}{35}>2.
\]
They control, respectively, the number of source cells, the
perturbation required in Lemma~\ref{lem:perturbed-determinant-en},
and the Farey count within each cell.

\emph{Determinant comparison.} For each integer $d$ with $2\leq d\leq D(A)$, let
$C_{d,J}^{\mathrm{pert}}$ be the constant in
Lemma~\ref{lem:perturbed-determinant-en}, and set
\begin{equation}\label{eq:uniform-det-constant-en}
K_{\det}(J,A)
:=
\max_{2\leq d\leq D(A)}C_{d,J}^{\mathrm{pert}}<\infty.
\end{equation}
Since $D(A)$ is fixed once $A$ is fixed, this maximum is taken over
a finite set of constants determined before the choice of $Q$ and of
the target block. In particular, $K_{\det}(J,A)$ is independent of
$Q$, $u$, $B$, and $H$.

By \eqref{eq:vertical-Q-s-en} and
\eqref{eq:s-kappa-gap-en},
\[
K_0Q^{-s}\leq K_0Q^{-\kappa N-1/5}.
\]
Hence $Q\geq K_0^5$ implies
\[
K_0Q^{-s}\leq Q^{-\kappa N}=\rho^N.
\]
This restriction is simultaneous for all large target blocks and is
independent of $u$, $B$, and $H$.

Suppose first that $C$ contains at least $N$ resonant centers. Choose
$N$ distinct centers and their fixed witnessing pairs, say
$(x_r,y_r)$, $1\leq r\leq N$, and define
\[
\Delta
:=
\det\bigl(x_r^iy_r^j\bigr)_{
\substack{1\leq r\leq N\\0\leq i\leq d,\ j\in\{0,1\}}}.
\]
All the $x_r$ lie in the closure of a cell of length at most $\rho$.
Lemma~\ref{lem:perturbed-determinant-en} and the definition of
$\kappa$ give
\begin{equation}\label{eq:det-upper-high-en}
|\Delta|
\leq
K_{\det}(J,A)\rho^{N(N-1)/2}
=
K_{\det}(J,A)Q^{-\frac32N(d+u)}.
\end{equation}
If $\Delta\neq0$, then $q_r<2Q$ and $b_r<2B$, so
Lemma~\ref{lem:det-lower-en} gives
\begin{equation}\label{eq:det-lower-high-en}
|\Delta|
\geq
(2Q)^{-dN}(2B)^{-N}
=
2^{-N(d+1)}Q^{-N(d+u)}.
\end{equation}
The two bounds are incompatible whenever
\begin{equation}\label{eq:det-incompatibility-en}
Q^{N(d+u)/2}
>
K_{\det}(J,A)2^{N(d+1)}.
\end{equation}
This threshold is uniform in the target block. Indeed,
$d\geq2$, $u\geq1/5$, and $N\geq6$, so
\[
\frac{N(d+u)}{2}
\geq\frac{6(2+1/5)}{2}
=\frac{33}{5}>6.
\]
Since $N(d+1)=2(d+1)^2$, it is enough to impose
\[
Q^6
>
K_{\det}(J,A)
\max_{2\leq d\leq D(A)}2^{2(d+1)^2}.
\]
This condition depends only on $J$ and $A$. Thus, for all sufficiently
large $Q$, the determinant associated with every choice of $N$
distinct resonant centers in every cell is zero.

Let $M_C$ be the number of resonant centers in $C$ for the present
target block. If $M_C<N$, then
\[
M_C<N=2(d+1)\leq2(D(A)+1),
\]
so $M_C=O_A(1)$. Assume that $M_C\ge N$ and enumerate the centers together with their
chosen witnesses as
\[
(x_r,y_r),\qquad 1\le r\le M_C,
\]
where $x_r=p_r/q_r$ and $y_r=a_r/b_r$. Consider the $M_C\times N$ evaluation matrix
\begin{equation}\label{eq:evaluation-matrix-en}
\mathcal M_C
:=
\bigl(x_r^iy_r^j\bigr)_{
\substack{1\leq r\leq M_C\\0\leq i\leq d,\ j\in\{0,1\}}}.
\end{equation}
Since the preceding choice of $N$ distinct centers was arbitrary, every
$N\times N$ minor of $\mathcal M_C$ vanishes. Hence
$\rank\mathcal M_C\leq N-1$, so the right kernel of $\mathcal M_C$ in
$\R^N$ is nontrivial. Choose a unit vector $(c_{i0},c_{i1})$ in this
right kernel and define
\begin{equation}\label{eq:auxiliary-polynomial-en}
P_C(X,Y)
:=
A_C(X)+B_C(X)Y
:=
\sum_{i=0}^dc_{i0}X^i
+
\sum_{i=0}^dc_{i1}X^iY,
\qquad
\norm{P_C}_2=1.
\end{equation}
Then
\begin{equation}\label{eq:P-vanishes-pairs-en}
P_C(x_r,y_r)=0
\qquad(1\leq r\leq M_C).
\end{equation}

\emph{Counting in one cell.} Put
\[
X_*:=\max\set{1,\max_{x\in J}|x|}
\]
and
\begin{equation}\label{eq:L-A-en}
L_{J,A}
:=
\max_{2\leq d\leq D(A)}
\left(\sum_{i=0}^dX_*^{2i}\right)^{1/2}<\infty.
\end{equation}
The coefficient normalization in
\eqref{eq:auxiliary-polynomial-en} and Cauchy--Schwarz give
$|B_C(x)|\leq L_{J,A}$ for $x\in J$. Since $P_C$ is linear in $Y$,
\eqref{eq:P-vanishes-pairs-en} and
\eqref{eq:vertical-Q-s-en} imply
\begin{equation}\label{eq:P-small-graph-en}
\abs{P_C(x_r,f(x_r))}
=
|B_C(x_r)|\abs{f(x_r)-y_r}
\leq
L_{J,A}K_0Q^{-s}.
\end{equation}

For $2\leq d\leq D(A)$, let
$C_{d,J}$, $R_{d,J}$, and $\eps_{d,J}$ be supplied by
Corollary~\ref{cor:polynomial-sublevel-en}, and set
\begin{equation}\label{eq:uniform-sublevel-data-en}
\begin{aligned}
C_{J,A}^{\mathrm{sub}}
&:=\max_{2\leq d\leq D(A)}C_{d,J},\\
R_{J,A}^{\mathrm{sub}}
&:=\max_{2\leq d\leq D(A)}R_{d,J},\\
\eps_{J,A}^{\mathrm{sub}}
&:=\min_{2\leq d\leq D(A)}\eps_{d,J}>0.
\end{aligned}
\end{equation}
These extrema are finite. Since $s\geq97/5$,
\[
L_{J,A}K_0Q^{-s}\leq L_{J,A}K_0Q^{-97/5},
\]
so, after increasing the lower threshold for $Q$ depending only on
$J$ and $A$, we have
\begin{equation}\label{eq:sublevel-smallness-en}
L_{J,A}K_0Q^{-s}<\eps_{J,A}^{\mathrm{sub}}.
\end{equation}
This threshold is again uniform over all large target blocks.
Define
\[
S_C
:=
\set{x\in J:
\abs{P_C(x,f(x))}\leq L_{J,A}K_0Q^{-s}}
\cap\overline C.
\]
Intersecting the sublevel set from
Corollary~\ref{cor:polynomial-sublevel-en} with the interval
$\overline C$ does not increase its number of connected components.
Every resonant center assigned to $C$ belongs to $S_C$. By
Corollary~\ref{cor:polynomial-sublevel-en},
\begin{equation}\label{eq:sublevel-length-high-en}
|S_C|\ll_{J,A}Q^{-s/(N-1)},
\end{equation}
and $S_C$ is a union of at most
$R_{J,A}^{\mathrm{sub}}$ intervals. From
\eqref{eq:s-lower-en} and \eqref{eq:N-upper-en},
\begin{equation}\label{eq:s-over-N-en}
\frac{s}{N-1}\geq\frac{97}{35}>2.
\end{equation}
The resonant centers are distinct reduced fractions with denominators
in $[Q,2Q)$. By \eqref{eq:sublevel-length-high-en} and
\eqref{eq:s-over-N-en},
\[
Q^2|S_C|
\ll_{J,A}
Q^{2-s/(N-1)}
\leq Q^{2-97/35}
=Q^{-27/35}.
\]
Lemma~\ref{lem:rational-union} therefore gives
\begin{align}
M_C
&\leq4Q^2|S_C|+R_{J,A}^{\mathrm{sub}}\notag\\
&\ll_{J,A}Q^{-27/35}+1
\ll_{J,A}1.
\label{eq:centers-per-cell-en}
\end{align}
Together with the case $M_C<N$, this bound holds uniformly for every
cell and every large target block.

\emph{Summation over cells and target blocks.} By \eqref{eq:number-cells-en}, there are $O_J(Q^\kappa)$ source
cells, and \eqref{eq:centers-per-cell-en} gives $O_{J,A}(1)$ resonant
centers in each. Hence, by \eqref{eq:kappa-bound-en}, one large target
block contributes
\begin{equation}\label{eq:high-one-block-en}
O_{J,A}(Q^\kappa)=O_{J,A}(Q^{33/20}).
\end{equation}
The total number of dyadic target blocks is at most
\[
1+\log_2\frac{T(Q,A)}{H}
\leq
1+\frac{A}{97}\log_2Q
=O_A(\log Q).
\]
Since $\log Q\ll Q^{1/20}$ for $Q\geq2$, summing
\eqref{eq:high-one-block-en} gives
\begin{equation}\label{eq:high-all-blocks-en}
\sum_{B_k\geq Q^v}\#\mathcal D_k
=O_{J,A}(Q^{33/20}\log Q)
=O_{J,A}(Q^{17/10}).
\end{equation}

Combining \eqref{eq:D-covered-by-blocks-en},
\eqref{eq:low-final-en}, and \eqref{eq:high-all-blocks-en}, we obtain
\eqref{eq:scarcity-bound-en}. We finally record the uniformity. For
fixed $J$ and $A$, only the finite range $2\leq d\leq D(A)$ occurs,
and the bounds $u\geq1/5$ and $s-\kappa N\geq1/5$ make all lower
restrictions on $Q$ in the large-height argument uniform in the target
block. Thus none depends on $H$, $u$, $B$, the cell, or the witnessing
pairs. Taking their maximum defines $Q_0(J,A)$, and enlarging the final
implied constant defines $C(J,A)$.
\end{proof}


\section{From two-height counting to rigidity}
\label{sec:fusion}

\subsection{A safe center}

The counting estimate becomes effective once its quadratic main term
is smaller than the supply of reduced source fractions. For the
deduction below, the precise exponent $17/10$ is not essential: what
matters is that the second term in Theorem~\ref{prop:scarcity-en} is $o(Q^2)$ as
$Q\to\infty$, uniformly in the lower target cutoff $H$, while
$H^{2-\tau}\to0$ as $H\to\infty$. These two features produce safe
rational centers at every sufficiently large source scale.

\begin{lemma}[A safe center in one source block]
\label{lem:safe-center-en}
Assume the hypotheses and notation of
Theorem~\ref{prop:scarcity-en}. Let $H\geq2$ be an integer and
suppose in addition that
\begin{equation}\label{eq:tail-smallness-en}
C_{\mathrm{low}}H^{2-\tau}
\leq
\frac{c_{\mathrm F}}4|J|.
\end{equation}
Then, for every sufficiently large integer $Q$, there exists a reduced
fraction
\[
r=\frac pq\in J,
\qquad Q\leq q<2Q,
\]
such that
\begin{equation}\label{eq:safe-center-en}
\abs{f(r)-a/b}
>
2b^{-\tau}+4M_*Q^{-A}
\end{equation}
for every $a\in\Z$ and every integer $b$ satisfying
\[
H\leq b<T(Q,A).
\]
\end{lemma}

\begin{proof}
By Lemma~\ref{lem:farey-lower-en}, for every sufficiently large $Q$,
the interval $J$ contains at least
\[
c_{\mathrm F}|J|Q^2
\]
reduced fractions with denominator in $[Q,2Q)$. On the other hand,
Theorem~\ref{prop:scarcity-en} and
\eqref{eq:tail-smallness-en} give
\[
\#\mathcal D(J;Q,A,H)
\leq
\frac{c_{\mathrm F}}4|J|Q^2+C(J,A)Q^{17/10}.
\]
Since $17/10<2$ and $J,A$ are fixed, there exists
$Q_1(J,A)$ such that, for every $Q\geq Q_1(J,A)$,
\[
C(J,A)Q^{17/10}
\leq
\frac{c_{\mathrm F}}4|J|Q^2.
\]
Thus $\mathcal D(J;Q,A,H)$ does not exhaust the reduced fractions
$p/q\in J$ with $Q\leq q<2Q$. Hence any such fraction outside
$\mathcal D(J;Q,A,H)$ satisfies \eqref{eq:safe-center-en} by definition.
\end{proof}

\subsection{Fusion and proof of Theorem~\ref{thm:intro-local}}

We first construct one point with the stated eventual lower bound.
The binary refinement at the end yields a Cantor set with the same
lower denominator cutoff.

Fix a nonempty open interval $V\subseteq U$. Since $f$ is not rational, it is nonconstant, so $f'\not\equiv0$. As $f'$ is real-analytic, its zeros are isolated. We may therefore choose a
nondegenerate compact interval
\begin{equation}\label{eq:I-star-choice-en}
I_*\Subset V
\end{equation}
on which $f'$ does not vanish. By continuity and compactness, there are
constants $m_*,M_*>0$ such that
\[
0<m_*\leq\abs{f'(x)}\leq M_*
\qquad(x\in I_*),
\]
as required in \eqref{eq:derivative-bounds-en}.

For $n\geq0$, put
\begin{equation}\label{eq:A-n-en}
A_n:=n+3
\end{equation}
and define
\begin{equation}\label{eq:Z-n-en}
\mathcal Z_n
:=
\bigcup_{2\leq d\leq D(A_n)}
\set{x\in I_*:W_d(x)=0}.
\end{equation}
Each $\mathcal Z_n$ is finite by
Corollary~\ref{cor:wronskian-finite-en}. Since $D(A)$ is nondecreasing,
\begin{equation}\label{eq:Z-n-nested-en}
\mathcal Z_n\subseteq\mathcal Z_{n+1}.
\end{equation}
Choose a nondegenerate compact interval
\begin{equation}\label{eq:I0-en}
I_0\subseteq\Int(I_*)\setminus\mathcal Z_0.
\end{equation}

For a nondegenerate compact interval $I=[\alpha,\beta]$, let
\begin{equation}\label{eq:middle-third-en}
J(I)
:=
\left[
\alpha+\frac{\beta-\alpha}{3},
\beta-\frac{\beta-\alpha}{3}
\right]
\end{equation}
be its middle third. Whenever $I_n$ is defined, write
$J_n:=J(I_n)$. Then
\begin{equation}\label{eq:middle-third-length-en}
|J_n|=\frac13|I_n|,
\qquad
\dist(J_n,\R\setminus I_n)=\frac13|I_n|.
\end{equation}

Choose an integer $T_0\geq2$ such that
\begin{equation}\label{eq:initial-tail-en}
C_{\mathrm{low}}T_0^{2-\tau}
\leq
\frac{c_{\mathrm F}}4|J_0|.
\end{equation}
This is possible because $2-\tau=-98<0$.

We construct nondegenerate compact intervals $I_n$, strictly
increasing positive integers $T_n$, and reduced fractions
$r_n=p_n/q_n$ satisfying the following properties whenever the objects
involved are defined:
\begin{enumerate}[label=\textup{(F\arabic*)},leftmargin=3em]
\item $I_{n+1}\subset\Int(I_n)$;
\item $I_n\cap\mathcal Z_n=\varnothing$;
\item $T_{n+1}>T_n$;
\item $q_n>2q_{n-1}$ for $n\geq1$;
\item for every $x\in I_{n+1}$,
\begin{equation}\label{eq:source-invariant-en}
0<\abs{x-r_n}<q_n^{-A_n};
\end{equation}
\item for every $x\in I_{n+1}$, every $a\in\Z$, and every integer $b$
with $T_n\leq b<T_{n+1}$,
\begin{equation}\label{eq:target-invariant-en}
\abs{f(x)-a/b}>b^{-\tau};
\end{equation}
\item
\begin{equation}\label{eq:tail-invariant-en}
C_{\mathrm{low}}T_n^{2-\tau}
\leq
\frac{c_{\mathrm F}}4|J_n|.
\end{equation}
\end{enumerate}
Conditions (F5) and (F6) impose source approximation and target
avoidance. Conditions (F2) and (F7) ensure the hypotheses needed
to choose the next safe center.

For clarity, we write \textup{(F1)}$_n$, \textup{(F3)}$_n$,
\textup{(F5)}$_n$, and \textup{(F6)}$_n$ for the corresponding
properties linking stage $n$ to stage $n+1$, while
\textup{(F2)}$_n$ and \textup{(F7)}$_n$ refer to the properties at
stage $n$ itself; similarly, \textup{(F4)}$_n$ is understood for
$n\geq1$.

The initial choices $I_0$ and $T_0$ satisfy
\textup{(F2)}$_0$ and \textup{(F7)}$_0$ by
\eqref{eq:I0-en} and \eqref{eq:initial-tail-en}. Suppose that $I_0,\ldots,I_n$, $T_0,\ldots,T_n$, and, when
$n\geq1$, $r_0,\ldots,r_{n-1}$ have been constructed and satisfy
all applicable induction hypotheses. We choose
$Q_n,r_n,T_{n+1},I_{n+1}$ in this order and verify the required
conditions at the next stage.

Since $J_n\subseteq I_n$ and
\textup{(F2)}$_n$ gives $I_n\cap\mathcal Z_n=\varnothing$,
\[
W_d(x)\neq0
\qquad(x\in J_n,\ 2\leq d\leq D(A_n)).
\]
Thus Theorem~\ref{prop:scarcity-en} applies on $J_n$ with
$A=A_n$, and \textup{(F7)}$_n$ is precisely the additional hypothesis
\eqref{eq:tail-smallness-en} of Lemma~\ref{lem:safe-center-en}, with
$H=T_n$.

Before choosing the source height, define
\begin{equation}\label{eq:Lambda-next-en}
\Lambda_{n+1}:=\#\mathcal Z_{n+1}+2.
\end{equation}
Notice that $\mathcal Z_{n+1}$, and hence $\Lambda_{n+1}$, is already
fixed at this point, independently of all objects to be chosen at
stage $n$.

Choose an integer $Q_n$ sufficiently large that the conclusion of
Lemma~\ref{lem:safe-center-en} holds and, simultaneously,
\begin{enumerate}[label=\textup{(Q\arabic*)},leftmargin=3em]
\item $Q_n>2q_{n-1}$ if $n\geq1$;
\item
\[
(2Q_n)^{-A_n}<\frac{|I_n|}{3};
\]
\item
\[
Q_n^{A_n/97}\geq2T_n+2;
\]
\item
\begin{equation}\label{eq:Q4-en}
3C_{\mathrm{low}}\Lambda_{n+1}
2^{A_n+\tau-2}Q_n^{-A_n/97}
\leq
\frac{c_{\mathrm F}}4.
\end{equation}
\end{enumerate}
Such a choice is possible. Indeed, at this stage $J_n$, $A_n$,
$T_n$, $\Lambda_{n+1}$, and, when $n\geq1$, $q_{n-1}$ are fixed, as
are all constants and thresholds in
Lemma~\ref{lem:safe-center-en}. Conditions
\textup{(Q1)}--\textup{(Q3)} impose only finitely many lower bounds on
$Q_n$, while the left-hand side of \eqref{eq:Q4-en} tends to zero as
$Q_n\to\infty$.

Choose a reduced safe fraction
\begin{equation}\label{eq:r-n-choice-en}
r_n=\frac{p_n}{q_n}\in J_n,
\qquad Q_n\leq q_n<2Q_n.
\end{equation}
Set
\begin{equation}\label{eq:R-T-next-en}
R_n:=(2Q_n)^{-A_n},
\qquad
T_{n+1}:=\left\lfloor Q_n^{A_n/97}\right\rfloor,
\end{equation}
and define
\begin{equation}\label{eq:K-next-en}
K_{n+1}:=(r_n-R_n,r_n+R_n).
\end{equation}
Condition \textup{(Q3)} gives $T_{n+1}\geq2T_n+2>T_n$, proving
\textup{(F3)}$_n$. Since $r_n\in J_n$ and
$\dist(J_n,\R\setminus I_n)=|I_n|/3$, condition \textup{(Q2)} gives
\[
K_{n+1}\subset\Int(I_n).
\]

We transfer the safety of $r_n$ to $K_{n+1}$. If
$T_n\leq b<T_{n+1}$, then
\[
b<T_{n+1}\leq Q_n^{A_n/97}=T(Q_n,A_n).
\]
Therefore the safety of $r_n$ gives, for every $a\in\Z$,
\begin{equation}\label{eq:safe-at-center-en}
\abs{f(r_n)-a/b}
>
2b^{-\tau}+4M_*Q_n^{-A_n}.
\end{equation}
For $x\in K_{n+1}$, the mean-value theorem and
\eqref{eq:derivative-bounds-en} yield
\[
\abs{f(x)-f(r_n)}
\leq M_*\abs{x-r_n}
<M_*R_n
\leq M_*Q_n^{-A_n}.
\]
Consequently,
\begin{equation}\label{eq:avoidance-on-K-en}
\abs{f(x)-a/b}
>
2b^{-\tau}+3M_*Q_n^{-A_n}
>b^{-\tau}.
\end{equation}

It remains to choose $I_{n+1}$ inside $K_{n+1}$. For every $k<n$,
property \textup{(F5)}$_k$ gives
$r_k\notin I_{k+1}$, and repeated use of \textup{(F1)} gives
$I_n\subseteq I_{k+1}$. Hence no previously chosen center $r_k$,
$k<n$, lies in $K_{n+1}$. It is therefore enough to remove
\begin{equation}\label{eq:forbidden-next-en}
F_{n+1}:=\mathcal Z_{n+1}\cup\{r_n\}.
\end{equation}
The set $F_{n+1}\cap K_{n+1}$ has at most
$\#\mathcal Z_{n+1}+1=\Lambda_{n+1}-1$ points. Hence
$K_{n+1}\setminus F_{n+1}$ has at most $\Lambda_{n+1}$ connected
components. Their total length is $2R_n$, so one component, say
$(u_{n+1},v_{n+1})$, has length at least
$2R_n/\Lambda_{n+1}$. Put
\[
s_n:=\frac{R_n}{\Lambda_{n+1}}
\]
and choose
\begin{equation}\label{eq:I-next-choice-en}
I_{n+1}
:=
\left[
 u_{n+1}+\frac{s_n}{2},
 u_{n+1}+\frac{3s_n}{2}
\right].
\end{equation}
Then $I_{n+1}\Subset(u_{n+1},v_{n+1})$ and
\begin{equation}\label{eq:I-next-length-en}
|I_{n+1}|
=s_n
=\frac{R_n}{\Lambda_{n+1}}
=\frac{(2Q_n)^{-A_n}}{\Lambda_{n+1}}.
\end{equation}
Thus \textup{(F1)}$_n$ and \textup{(F2)}$_{n+1}$ hold.
Furthermore, $r_n\notin I_{n+1}$ and
$I_{n+1}\subset K_{n+1}$, so for every $x\in I_{n+1}$,
\[
0<\abs{x-r_n}<R_n\leq q_n^{-A_n},
\]
where the last inequality follows from $q_n<2Q_n$. This proves
\textup{(F5)}$_n$. Property \textup{(F6)}$_n$ follows from
\eqref{eq:avoidance-on-K-en}. If $n\geq1$, then
$q_n\geq Q_n>2q_{n-1}$ by \textup{(Q1)}, proving
\textup{(F4)}$_n$.

It remains to verify \textup{(F7)}$_{n+1}$. From
\eqref{eq:I-next-length-en} and
\eqref{eq:middle-third-length-en},
\begin{equation}\label{eq:J-next-lower-en}
|J_{n+1}|
=\frac{(2Q_n)^{-A_n}}{3\Lambda_{n+1}}.
\end{equation}
Condition \textup{(Q3)} implies $Q_n^{A_n/97}\geq2$. Since
$\lfloor X\rfloor\geq X/2$ for $X\geq2$,
\[
T_{n+1}\geq\frac12Q_n^{A_n/97}.
\]
Because $2-\tau<0$,
\begin{equation}\label{eq:T-tail-upper-en}
T_{n+1}^{2-\tau}
\leq
2^{\tau-2}Q_n^{-A_n(\tau-2)/97}.
\end{equation}
Combining \eqref{eq:J-next-lower-en} and \eqref{eq:T-tail-upper-en} gives
\begin{align}
\frac{C_{\mathrm{low}}T_{n+1}^{2-\tau}}{|J_{n+1}|}
&\leq
3C_{\mathrm{low}}\Lambda_{n+1}
2^{A_n+\tau-2}
Q_n^{A_n-A_n(\tau-2)/97}\notag\\
&=
3C_{\mathrm{low}}\Lambda_{n+1}
2^{A_n+\tau-2}Q_n^{-A_n/97}\notag\\
&\leq\frac{c_{\mathrm F}}4,
\label{eq:tail-ratio-next-en}
\end{align}
where we used $97=\tau-3$ and \textup{(Q4)}. The negative exponent $-A_n/97$ ensures that the contribution
of the next lower target cutoff becomes small relative to the
length of the next source interval. This is
\textup{(F7)}$_{n+1}$. We have therefore completed the inductive step,
and hence the construction.

By \textup{(F1)}, the intervals $I_n$ are nested, nonempty, and
compact. Moreover, for $n\geq1$,
\[
Q_n>2q_{n-1}\geq2Q_{n-1},
\]
so $Q_n\to\infty$. Since $I_{n+1}\subset K_{n+1}$,
\[
\diam(I_{n+1})\leq2R_n=2(2Q_n)^{-A_n}\longrightarrow0.
\]
The nested-interval theorem therefore gives a unique point
\begin{equation}\label{eq:xi-intersection-en}
\set{\xi}=\bigcap_{n=0}^{\infty}I_n.
\end{equation}
Property \textup{(F5)} yields
\begin{equation}\label{eq:xi-rn-en}
0<\abs{\xi-r_n}<q_n^{-A_n}
\qquad(n\geq0).
\end{equation}
The denominators $q_n$ tend to infinity. Given $M\geq1$, we have
$A_n=n+3\geq M$ for all sufficiently large $n$, and hence
\[
0<\abs{\xi-r_n}<q_n^{-M}
\]
for infinitely many $n$. Thus $\xi\in\Lio$.

By \textup{(F3)}, the positive integers $T_n$ are strictly increasing
and tend to infinity. The half-open blocks
$[T_n,T_{n+1})\cap\Z_{>0}$ therefore partition all integers $b\geq T_0$.
For such a $b$, let $n$ be the unique index with
$T_n\leq b<T_{n+1}$. Since $\xi\in I_{n+1}$, property \textup{(F6)}
gives, for every $a\in\Z$,
\begin{equation}\label{eq:image-lower-en}
\abs{f(\xi)-a/b}>b^{-\tau}=b^{-100}.
\end{equation}
This is the asserted eventual lower bound.

For every $\lambda>\tau$, the inequality
$0<|f(\xi)-a/b|<b^{-\lambda}$ can hold only for $b<T_0$,
and for each such $b$ there are only finitely many possible
numerators. Hence $\mu(f(\xi))\leq\tau=100$. Finally, $\xi\in I_0\subseteq I_*\subseteq V$.

\medskip
\noindent\emph{Cantor refinement.}
Repeat the induction on a binary tree, with the same initial interval
$I_0$ and cutoff $T_0$. At each node of depth $n$, use $A_n=n+3$ and
replace the factor $3$ in \textup{(Q4)} by $6$. All other choices up to
the component $(u_{n+1},v_{n+1})$ remain unchanged. That component has
length at least $2s_n$, where $s_n=R_n/\Lambda_{n+1}$. Choose the two
children to be
\[
\left[u_{n+1}+\frac{s_n}{4},u_{n+1}+\frac{3s_n}{4}\right],
\qquad
\left[u_{n+1}+\frac{5s_n}{4},u_{n+1}+\frac{7s_n}{4}\right].
\]
They are disjoint compact subintervals of the component, each of length
$s_n/2$. Both inherit \textup{(F2)}, \textup{(F5)}, and \textup{(F6)}.
Their middle thirds have length $R_n/(6\Lambda_{n+1})$, so the same
calculation as in \eqref{eq:tail-ratio-next-en}, with $3$ replaced by
$6$, proves \textup{(F7)} for each child.

At each level there are only finitely many nodes, and at each node all
data are fixed before its source height is chosen. Thus the induction
continues at every node. Along each branch, \textup{(Q1)} implies
$Q_n\geq2^{n+1}$ after taking the initial height at least $2$; hence
the diameters tend to zero uniformly over all branches. The intersection
of the successive finite unions of node intervals is therefore a
nonempty compact perfect totally disconnected set $K\subseteq V$.
Indeed, each branch determines one point, distinct branches separate
into disjoint children, and branching at arbitrarily large depths
precludes isolated points.

For every $\xi\in K$, its branch satisfies the source approximation
estimates, so $\xi\in\Lio$. The target blocks along that branch cover
all integers $b\geq T_0$, and \textup{(F6)} gives
$\abs{f(\xi)-a/b}>b^{-\tau}$ for every $a\in\Z$ and every such $b$.
Thus the same cutoff $b_0=T_0$ works for every point of $K$, completing
the proof of Theorem~\ref{thm:intro-local}.
\qed

\subsection{Proof of Corollary~\ref{cor:intro-rigidity}}

If the conclusion were false, then $f$ would not be the restriction to
$U$ of any rational function in $\R(x)$. Theorem~\ref{thm:intro-local}
would then produce a point $\xi\in\Lio\cap U$ for which
$\mu(f(\xi))\leq100$. Hence $f(\xi)\notin\Lio$, contradicting
$f(\Lio\cap U)\subseteq\Lio$.
\qed

\subsection{Proof of Theorem~\ref{thm:intro-mahler}}

Let $F:\C\to\C$ be entire and suppose that
$F(\Lio)\subseteq\Lio$. Since $\Lio$ is dense in $\R$, every
$x\in\R$ is the limit of a sequence $(\xi_n)$ in $\Lio$. The values
$F(\xi_n)$ are real, and continuity gives
\[
F(x)=\lim_{n\to\infty}F(\xi_n)\in\R.
\]
Thus $F(\R)\subseteq\R$.

Apply Corollary~\ref{cor:intro-rigidity} to the real-analytic function
$F|_{\R}$. There are polynomials $P,Q\in\R[z]$, with $Q\neq0$ and with no
common zero in $\C$, such that
\[
F(x)=\frac{P(x)}{Q(x)}
\qquad(x\in\R).
\]
The entire function
\[
z\longmapsto Q(z)F(z)-P(z)
\]
vanishes on $\R$ and hence vanishes identically on $\C$. If $Q$ were
nonconstant, let $\zeta\in\C$ be a zero of $Q$. The identity
$QF-P\equiv0$ would give $P(\zeta)=0$, contrary to the coprimality of
$P$ and $Q$. Hence $Q$ is constant, and therefore $F\in\R[z]$.
\qed

\section*{Use of artificial intelligence tools}
OpenAI's ChatGPT was used for English-language editing and organization
of the exposition. The author is responsible for the final text and for
the correctness of all statements and proofs.

\end{document}